\documentclass[11pt]{article}
\usepackage{amsmath,amsthm, amssymb, amsfonts,accents,nccmath}
\usepackage{enumitem}
\usepackage{array}
\usepackage{lipsum}
\usepackage{mathtools}
\usepackage[toc,page]{appendix}
\usepackage[english]{babel}
\usepackage{dsfont}
\usepackage{booktabs}
\usepackage[linesnumbered,commentsnumbered,ruled,vlined]{algorithm2e}
\usepackage{tikz}
\usepackage{tikz-network}
\usepackage[utf8]{inputenc}
\usepackage[english]{babel}
\DeclarePairedDelimiter\ceil{\lceil}{\rceil}

\usepackage{caption}
\usepackage{subcaption}

\newcommand\restr[2]{{
  \left.\kern-\nulldelimiterspace 
  #1 
  \vphantom{\big|} 
  \right|_{#2} 
  }}

\usepackage{hyperref}
\hypersetup{
    colorlinks=true,
    linkcolor=blue,
    filecolor=darkmagenta,      
    urlcolor=cyan,
    citecolor=red
} 
\makeatletter
\def\BState{\State\hskip-\ALG@thistlm}
\makeatother
\numberwithin{equation}{section}

\newtheorem{theorem}{Theorem}[section]

\newtheorem{lem}[theorem]{Lemma}
\newtheorem{prop}[theorem]{Proposition}

\newtheorem{rem}[theorem]{Remark}

\newtheorem{obs}[theorem]{Observation}

\title{ On the spectral radius of bi-block graphs with  given \\ independence number $\alpha$}
\author{Joyentanuj Das$^*$   \quad and \quad Sumit Mohanty\footnote{School of Mathematics, IISER Thiruvananthapuram, Maruthamala P.O., Vithura, 
Thiruvananthapuram,\newline \indent   Kerala- 695 551, India.
 \newline \indent Emails:
joyentanuj16@iisertvm.ac.in,  \ sumit@iisertvm.ac.in, sumitmath@gmail.com }}

\date{}

\begin{document}

\maketitle

\begin{abstract}
A connected graph is called a bi-block graph if each of its blocks is a complete bipartite graph. Let $\mathcal{B}(\mathbf{k}, \alpha)$ be the class of bi-block graph on $\mathbf{k}$ vertices with  given independence number $\alpha$. It is easy to see that every bi-block graph is a bipartite graph. For a bipartite graph $G$ on $\mathbf{k}$ vertices, the independence number  $\alpha(G)$ satisfies $\ceil*{\frac{\mathbf{k}}{2}} \leq \alpha(G) \leq \mathbf{k}-1$.  In this article, we prove that the maximum spectral radius $\rho(G)$ among  all graphs $G$ in $\mathcal{B}(\mathbf{k}, \alpha)$,  is uniquely attained for the complete bipartite graph $K_{\alpha, \mathbf{k}-\alpha}$.

\end{abstract}



\section{Introduction}
Let $G=(V(G),E(G))$  be a finite, simple, connected graph with $V(G)$ as the set of vertices and $E(G)$ as the set of edges in $G$. We simply write $G=(V,E)$ if there is no scope of confusion. We write $u\sim v$ to indicate that the vertices $u$ and $v$ are adjacent in $G$. For $v\in V$, the degree of the vertex $v$ is denoted by $d_G(v)$, equals the number of vertices in $G$ that are adjacent to $v$. A graph $H$ is said to be a subgraph of $G$ if $V(H) \subset V(G)$ and $E(H) \subset E(G)$. For any subset $S \subset V (G)$, a subgraph $H$ of $G$ is said to be an induced subgraph with vertex set $S$ if $H$ is a maximal subgraph of $G$ with vertex set $V(H)=S$. We use $|S|$ to denote the cardinality of the set $S$. A graph $G=(V,E)$ is said to be bipartite if the vertex set $V$ can be partitioned into two subsets $M$ and $N$ such that  $E\subset M\times N$. A bipartite graph with vertex partitions $M$ and $N$  is said to be complete bipartite graph if every vertex of $M$ is adjacent to every vertex of $N$ and moreover, if $|M| = m$  and $|N| = n$, then the  complete bipartite graph is denoted by  $K_{m,n}$.  To emphasize the vertex  partition $ M$ and  $N$, we  use  $K(M,N)$ to represent $K_{m,n}$ whenever $|M|=n$ and $|N|=n$.

 A vertex $v$ of a connected graph $G=(V,E)$ is a cut-vertex of $G$ if $G - v$ is disconnected. A block of the graph $G$ is a maximal connected subgraph of G that has no cut-vertex. A block is said to be a leaf block if its deletion does not disconnect the graph. Given two blocks  $F$ and $H$  of graph $G$ are said to be neighbours if they are connected via a cut-vertex.  We write $F\circledcirc H$ to represent the induced subgraph on the vertex set of two neighbouring blocks $F$ and $H$.  A connected graph is called a bi-block graph if each of its blocks is a complete bipartite graph (see Figure~\ref{fig:bi-block}).  For $v\in V$, the  block index of $v$ is denoted by $bi_G(v)$,  equals the number of blocks in $G$ that  contain the vertex  $v$. It is easy to see that if $v$ is not a cut-vertex, then  $bi_G(v)=1$. Also note that, the star  $K_{1,n}$ is bi-block graph with a central cut-vertex $v(say)$, with $bi_G(v)=d_G(v)=n$ and each of its blocks are edges. In this article, we consider the star $K_{1,n}$ as a complete bipartite graph instead of a bi-block graph.

 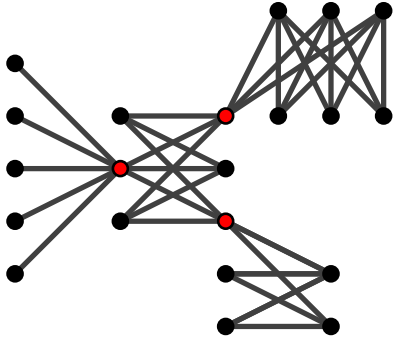
\begin{figure}[ht]
\centering
\begin{tikzpicture}[scale=0.7]
\Vertex[size=.1,color=black]{A} 
\Vertex[y=1,size=.1,color=red]{B} 
\Vertex[y=2,size=.1,color=black]{C}
\Vertex[x=2,y=1,size=.1,color=black]{F} 
\Vertex[x=2,y=2,size=.1,color=red]{G}
\Vertex[x=2,size=.1,color=red]{H}
\Vertex[x=2,y=-1,size=.1,color=black]{I} 
\Vertex[x=3,y=2,size=.1,color=black]{G1}
\Vertex[x=4,y=2,size=.1,color=black]{G2}
\Vertex[x=5,y=2,size=.1,color=black]{G3}
\Vertex[x=3,y=4,size=.1,color=black]{H1}
\Vertex[x=4,y=4,size=.1,color=black]{H2}
\Vertex[x=5,y=4,size=.1,color=black]{H3}
\Vertex[x=2,y=-2,size=.1,color=black]{I1} 
\Vertex[x=4,y=-1,size=.1,color=black]{J1} 
\Vertex[x=4,y=-2,size=.1,color=black]{J2}
\Vertex[x=-2,y=-1,size=.1,color=black]{K}
\Vertex[x=-2,size=.1,color=black]{K1}
\Vertex[x=-2,y=1,size=.1,color=black]{K2}
\Vertex[x=-2,y=2,size=.1,color=black]{K3}
\Vertex[x=-2,y=3,size=.1,color=black]{K4}

\Edge[lw=2pt](A)(F)
\Edge[lw=2pt](A)(G)
\Edge[lw=2pt](A)(H)
\Edge[lw=2pt](B)(F)
\Edge[lw=2pt](B)(G)
\Edge[lw=2pt](B)(H)
\Edge[lw=2pt](C)(F)
\Edge[lw=2pt](C)(G)
\Edge[lw=2pt](C)(H)
\Edge[lw=2pt](J1)(H)
\Edge[lw=2pt](J1)(I)
\Edge[lw=2pt](J1)(I1)
\Edge[lw=2pt](J2)(H)
\Edge[lw=2pt](J2)(I)
\Edge[lw=2pt](J2)(I1)
\Edge[lw=2pt](J1)(H)
\Edge[lw=2pt](J1)(I)
\Edge[lw=2pt](J1)(I1)
\Edge[lw=2pt](G)(H1)
\Edge[lw=2pt](G1)(H1)
\Edge[lw=2pt](G2)(H1)
\Edge[lw=2pt](G3)(H1)
\Edge[lw=2pt](G)(H2)
\Edge[lw=2pt](G1)(H2)
\Edge[lw=2pt](G2)(H2)
\Edge[lw=2pt](G3)(H2)
\Edge[lw=2pt](G)(H3)
\Edge[lw=2pt](G1)(H3)
\Edge[lw=2pt](G2)(H3)
\Edge[lw=2pt](G3)(H3)
\Edge[lw=2pt](B)(K)
\Edge[lw=2pt](B)(K1)
\Edge[lw=2pt](B)(K2)
\Edge[lw=2pt](B)(K3)
\Edge[lw=2pt](B)(K4)

\end{tikzpicture}
\caption{ bi-block graph  with blocks $K_{1,5}-K_{3,3}-K_{4,3}-K_{3,2}$} \label{fig:bi-block}
\end{figure}
 
A set $\mathcal{I}$ of vertices in a graph $G$ is an independent set if no pair of vertices of $\mathcal{I}$ are adjacent. The independence number of $ G$ is denoted by $\alpha(G)$,  equals the cardinality of the largest independent set in $G$. An independent set of cardinality $\alpha(G)$ is called an $\alpha(G)$-set. 

Let $G=(V,E)$ be a graph. The adjacency matrix  of $G$ is a $|V| \times |V|$ matrix $A(G) = [a_{xy}]_{x,y \in V}$, where $a_{xy}= 1 $  if $x\sim y$ and $0$ otherwise. For any column vector $X$ of order $|V|$, if $x_u$  represent the entry of $X$ corresponding to the vertex $u\in V$, then 
 $$X^t A(G)X=  2\sum_{u\sim w}x_{u}x_{w}, $$
 where $X^t$ denotes  the transpose of  $X$. For a connected graph $G$    on $\mathbf{k}\geq 2$ vertices, by Perron-Frobenius theorem, the spectral radius  $\rho(G)$ of $A(G)$ is a simple positive eigenvalue and  the associated eigenvector is entry-wise positive (for details see~\cite{Bapat}). We will refer to such an eigenvector as the
Perron vector of $G$. We simply write  $A$ and $\rho$ to denote the adjacency matrix and spectral radius, if there is no scope of confusion. Now we state a few known results on spectral radius useful in our subsequent calculations. By Min-max theorem, we have
$$\rho(G) = \max_{X\neq 0}\dfrac{X^t A(G) X}{X^tX}= \max_{X\neq 0}\dfrac{2\sum_{u\sim w}x_{u}x_{w}}{\sum_{u\in V}x_{u}^2}.$$
The following lemma gives a relation between spectral radius and   degree of vertices.
\begin{lem}\label{lem:sr_deg}\cite{Bapat}
For a graph $G$, if $\Delta(G)$ and $\delta(G)$ denote the maximum and the minimum of the vertex degrees of $G$, respectively,  then 
$\delta(G) \leq \rho(G) \leq \Delta(G)$.
\end{lem}
Given a graph $G=(V,E)$, for $x, y \in V(G)$ we will use  $G + xy$ to denote the graph obtained from $G$  by adding an edge $xy \notin E(G)$ and we have the following result.
\begin{lem}\label{lem:sr_edge}\cite{Bapat,Li}
If $G$ is a connected graph such that for $x, y \in V(G)$, $xy \notin E(G)$, then $\rho(G) < \rho(G+xy)$. 
\end{lem}

Problems related to maximization and minimization of spectral radius for a given class of graph is an active area in spectral graph theory and has been extensively studied (for example see \cite{Cri}-\cite{Xu}).  In particular, a few interesting articles related to maximization of the spectral radius with given independence number has been studied for trees and block graphs (for details see~\cite{Cri,Chu}). We are interested in maximizing spectral radius for bi-block graphs with given independence number.  Let $\mathcal{B}(\mathbf{k}, \alpha)$ be the class of bi-block graphs on $\mathbf{k}$ vertices with a given independence number $\alpha$. In this article, we prove that the maximum spectral radius $\rho(G)$, among  all graphs $G$ in $ \mathcal{B}(\mathbf{k}, \alpha)$  is uniquely attained for the complete bipartite graph $K_{\alpha, \mathbf{k}-\alpha}$.

\section{Main Results}

We begin with a few results that gives us some insight to dependency of the independence number of a bi-block graph and its leaf blocks. It is easy to see, if $G$ is a bipartite graph with vertex partition $M$ and $N$, then $\alpha(G)= \max\{|M|, |N|\}$. Since every bi-block graph is a bipartite graph, so given a bi-block graph $G$   on $\mathbf{k}$ vertices, the independence number $\alpha(G)$,  satisfies $\ceil*{\frac{\mathbf{k}}{2}} \leq \alpha(G) \leq \mathbf{k}-1$. 

Let $G$ be a  bi-block graph.  Let $H$  be any leaf block connected to the graph $G$ at a cut-vertex $v \in V(G)$ and $G-H$ be the graph obtained from $G$ by removing $H - v$. Given an $\alpha(G)$-set  $\mathcal{I}$,   we denote $$\restr{\mathcal{I}}{G-H} = \{u \in \mathcal{I} \mid u \in V(G-H) \}.$$ 
Note that, $\restr{\mathcal{I}}{G-H}$ is an independent set of the graph $G-H$ which need not be an  $\alpha(G-H)$-set and hence $\left|\restr{\mathcal{I}}{G-H}\right| \leq \alpha(G-H)$. Before providing a result that gives us a relation between  $\restr{\mathcal{I}}{G-H}$  and an $\alpha(G-H)$-set, we first prove the following lemma.
\begin{lem}\label{lem:2.1}
Let $G$ be a graph and $\mathcal{J}$ be an  $\alpha(G)$-set. For 
$v\in \mathcal{J}$,  if  $\mathcal{I}$ is an $\alpha(G-v)$-set and  $\mathcal{I}$ is not an $\alpha(G)$-set,  then $ |\mathcal{J}|= |\mathcal{I}| +1.$
\end{lem}
\begin{proof}
Since $\mathcal{I}$ is an  $\alpha(G-v)$-set and $\mathcal{J}-\{v\}$ is an independent set in $G-v$, so $|\mathcal{J}|-1 \leq |\mathcal{I}|.$ By hypothesis we have $|\mathcal{I}| < |\mathcal{J}|$ and hence $|\mathcal{J}|= |\mathcal{I}| +1.$\end{proof}

\begin{lem}\label{lem:ind_set}
Let $G$ be a  bi-block graph and  $H=K_{m,n}$  be any leaf block connected to the graph $G$ at a cut-vertex $v \in V(G)$ and $G-H$ be the graph obtained from $G$ by removing $H - v$. Let $\mathcal{I}$ be an $\alpha(G)$-set and $\restr{\mathcal{I}}{G-H}$ be defined as above. Then the following results hold.
\begin{enumerate}
\item[($i$)] If $v\in \mathcal{I},$ then $\restr{\mathcal{I}}{G-H}$ is an $\alpha(G-H)$-set.
\item[($ii$)] If $v\notin \mathcal{I}$ and $\restr{\mathcal{I}}{G-H}$ is not an $\alpha(G-H)$-set, then $v\in \mathcal{J}$ for all $\alpha(G-H)$-set $\mathcal{J}$.
\item[($iii$)] The independence number of $G-H$ is given by
$$
\alpha(G-H)=
\begin{cases}\vspace*{.2cm}
\left|\restr{\mathcal{I}}{G-H}\right| + 1 &  \textup{if } v \notin  \mathcal{I} \textup{ and } \restr{\mathcal{I}}{G-H}  \textup{ is not an }\alpha(G-H)\textup{-set},\\
\left|\restr{\mathcal{I}}{G-H}\right| & \text{otherwise}.
\end{cases}
$$
\end{enumerate}
\end{lem}
\begin{proof}
Let $H=K(M,N)$ where $|M|=m$, $|N|=n$ and  without loss of generality, let us assume $v\in M$. First, if $v\in \mathcal{I}$, then  $\alpha(G)=|\mathcal{I}|= \left|\restr{\mathcal{I}}{G-H}\right|+m-1$ and moreover,  $m\geq n$, else $\mathcal{L}= \left(\restr{\mathcal{I}}{G-H}- \{v\}\right)\cup N$ is an independent set in $G$ with $|\mathcal{L}|=\left|\restr{\mathcal{I}}{G-H}\right|+n-1 >\alpha(G)$, which is a contradiction.  In this case, if $\restr{\mathcal{I}}{G-H}$ is not an  $\alpha(G-H)$-set, then there exists  an  
independent set  $\mathcal{J}\subset V(G-H)$ such that $|\mathcal{J}| > \left|\restr{\mathcal{I}}{G-H}\right|.$ 
Thus, $\mathcal{L}= \mathcal{J}\cup (M-\{v\})$ is an independent set in $G$ and $|\mathcal{L}|=|\mathcal{J}|+m-1 > \left|\restr{\mathcal{I}}{G-H}\right| + m-1= \alpha(G),$  which is a contradiction. This proves part $(i).$ Next, let $v\notin \mathcal{I}$ and $\restr{\mathcal{I}}{G-H}$ is not an $\alpha(G-H)$-set. Then, 
\begin{equation*}
 \mathcal{I}=
\begin{cases}
 \restr{\mathcal{I}}{G-H} \cup (M-\{v\}) &\textup{ if } m-1>n,\\ 
\restr{\mathcal{I}}{G-H} \cup N &\textup{ if }  m-1<n,\\
 \restr{\mathcal{I}}{G-H} \cup (M-\{v\}) \textup{ or }  \restr{\mathcal{I}}{G-H} \cup N &\textup{ if } m-1=n,
\end{cases}
\end{equation*}
 and hence $\alpha(G)=|\mathcal{I}|= \left|\restr{\mathcal{I}}{G-H}\right|+\max \{m-1,n\}$. Suppose that there exists an $\alpha(G-H)$-set $\mathcal{J} (\subset V(G-H))$ such that $v\notin \mathcal{J}$. Then, if we consider $\mathcal{L} \subset V(G)$ such that
\begin{equation*}
\mathcal{L}=
\begin{cases}
 \mathcal{J} \cup (M-\{v\}) &\textup{ if } m-1>n,\\ 
\mathcal{J} \cup N &\textup{ if }  m-1<n,\\
\mathcal{J} \cup (M-\{v\}) \textup{ or } \mathcal{J} \cup N &\textup{ if } m-1=n,
\end{cases}
\end{equation*}
then $\mathcal{L}$ is  an independent set in $G$ and  $|\mathcal{L}|= |\mathcal{J}|+\max \{m-1,n\}$. Since  $\mathcal{J}$ is an $\alpha(G-H)$-set, so $|\mathcal{J}|> \left|\restr{\mathcal{I}}{G-H}\right|$ and hence  $|\mathcal{L}| >  \alpha(G)$,  which is a contradiction. This proves part $(ii).$ 

Furthermore, if $v\notin \mathcal{I}$ and $\restr{\mathcal{I}}{G-H}$ is not an $\alpha(G-H)$-set, then $\restr{\mathcal{I}}{G-H}$ is an $\alpha((G-H)-v)$-set, else there exists an independent set $\mathcal{J}\subset V((G-H)-v)$ such that $|\mathcal{J}|> \left|\restr{\mathcal{I}}{G-H}\right|$ and  hence proceeding similar to  part $(ii)$ leads to a contradiction. Next, in this case, if $\mathcal{L}$ is an $\alpha(G-H)$-set, then by part $(ii)$, we have $v\in \mathcal{L}$ and hence using Lemma~\ref{lem:2.1}, we get $\alpha(G-H)= |\mathcal{L}|= \left|\restr{\mathcal{I}}{G-H}\right| + 1.$ This completes the proof.
\end{proof}

The next result relates the independence number of a bi-block of a graph with its leaf blocks.

\begin{prop}\label{prop:alpha}
Let $G$ be a  bi-block graph and $H=K_{m,n}$  be any leaf block connected to the graph $G$ at a cut-vertex $v$ and $G-H$ be the graph obtained from $G$ by removing $H - v$. If $m \geq n$, then  $\alpha(G)$ equals to  either $ \alpha(G-H)+m$ or $\alpha(G-H)+m-1$.
\end{prop}
\begin{proof}
Let $H=K(M,N)$ where $|M|=m$, $|N|=n$  and $m\geq n$.  Let $\mathcal{I}$ be an $\alpha(G)$-set and $\restr{\mathcal{I}}{G-H}$ be as defined  before. We consider the following cases to complete the proof. First consider the case whenever $ v \notin  \mathcal{I} \text{ and } \restr{\mathcal{I}}{G-H} \mbox{ is an }\alpha(G-H)\mbox{-set}$. If $v\in M$ and $m>n$, then $\alpha(G) =  \left|\restr{\mathcal{I}}{G-H}\right| + m-1$. Otherwise  $\alpha(G) = \left|\restr{\mathcal{I}}{G-H}\right| + m$. Thus, by part ($iii$) of Lemma~\ref{lem:ind_set} the result is true.

Next, consider the case  $ v \notin  \mathcal{I} \text{ and } \restr{\mathcal{I}}{G-H}  \mbox{ is not an }\alpha(G-H)$-set. Let $\mathcal{J}$ be an $\alpha(G-H)$-set. Then $|\mathcal{J}|>\left|\restr{\mathcal{I}}{G-H}\right|$ and by part ($ii$) of Lemma~\ref{lem:ind_set}, we have $v\in \mathcal{J}$.    If $m>n$, then $v$   necessarily  belongs to $N$. Suppose on the contrary, let $m>n$ and $v\in M$. Then, $\mathcal{I}=\restr{\mathcal{I}}{G-H} \cup \left(M\setminus \{v\}\right)$ and hence $\alpha(G)=|\mathcal{I}|=\left|\restr{\mathcal{I}}{G-H}\right|+m-1.$ Note that,  $\mathcal{L} = \mathcal{J} \cup \left(M\setminus \{v\}\right)$ is an independent set of $G$ with $|\mathcal{L}| > \alpha(G)$, which leads to a contradiction.  Therefore,   $\alpha(G) = \left|\restr{\mathcal{I}}{G-H}\right| + m$ and by part ($iii$) of Lemma~\ref{lem:ind_set}, we have $\alpha(G) = \alpha(G-H)+m-1$.   If $m=n$, then 
$$\mathcal{I}=
\begin{cases}
\restr{\mathcal{I}}{G-H} \cup N & \mbox{ if } v \in M,\\
\restr{\mathcal{I}}{G-H} \cup M & \mbox{ if } v \in N.
\end{cases}$$
Thus, by part ($iii$) of Lemma~\ref{lem:ind_set}, we have $\alpha(G) = \alpha(G-H)+m-1$. 

Finally, let $v\in \mathcal{I}$. Then, by part ($i$) of Lemma~\ref{lem:ind_set} we know that $\restr{\mathcal{I}}{G-H} \mbox{ is an }\alpha(G-H)\mbox{-set}$. If $m>n$, then the cut-vertex $v$ necessarily belongs to $M$, else $\mathcal{L} = \left(\restr{\mathcal{I}}{G-H} \setminus \{v\}\right) \cup M$ is an independent set of $G$ with $|\mathcal{L}| > \alpha(G)$, which leads to a contradiction. Thus  $\alpha(G) = |\restr{\mathcal{I}}{G-H}| + m-1$ and by part ($iii$) of Lemma~\ref{lem:ind_set}, we get $\alpha(G) = \alpha(G-H)+m-1$. If $m=n$, then 
$$\mathcal{I}=
\begin{cases}
\restr{\mathcal{I}}{G-H} \cup \left(M\setminus \{v\}\right) & \mbox{ if } v \in M,\\
\restr{\mathcal{I}}{G-H} \cup \left(N\setminus \{v\}\right) & \mbox{ if } v \in N.
\end{cases}$$
Thus, by part ($iii$) of Lemma~\ref{lem:ind_set}, we have $\alpha(G) = \alpha(G-H)+m-1$. This completes the proof.
\end{proof}

Now we consider our main goal,  to maximize the spectral radius  for the class of bi-block graphs $ \mathcal{B}(\mathbf{k},\alpha)$. We begin with the result for bi-block graphs consisting of two-blocks. Before proceeding for the result,  we list a few identities as an observation below.

\begin{obs}\label{obs:1}
Let $G=(V,E)$ be a bi-block graph consisting of two blocks $F$ and $H$ connected by cut-vertex $v$, {\it i.e.,} $G= F\circledcirc H$. Let $F=K(P,Q)$ with $|P|=p$, $|Q|=q$ and $H=K(M,N)$ with $|M|=m$, $|N|=n$ such that $Q \cap M = \{v\}$. Let $A$ be the adjacency matrix of $G$ and $(\rho,X)$ be the eigen-pair corresponding to the spectral radius of $A$. Let $x_u$ denote the entry  of $X$ corresponding to the vertex $u\in V$.

Let  $q,m \geq  2$. Using $AX = \rho X$, we have $\rho x_u= \sum_{w\sim u}x_w= \sum_{w\in M}x_w$ for all $u \in N$. Thus $x_u$  is a constant, whenever 
$u \in N$ and we denote it by $a_n$. Using similar arguments, let us denote
\begin{equation}\label{eqn:1}
x_u=
\begin{cases} 
a_n & \mbox{ if } u \in N,\\
a_m & \mbox{ if } u \in M, u\neq v,\\
a_p & \mbox{ if } u \in P,\\
a_q & \mbox{ if } u \in Q, u\neq v.
\end{cases}
\end{equation}
Now using $AX = \rho X$, we have the following identities:
\begin{itemize}[noitemsep]
\item[\textup{(I$1$)}] $(q-1)a_q + x_v = \rho a_p$.
\item[\textup{(I$2$)}] $pa_p = \rho a_q$.
\item[\textup{(I$3$)}] $pa_p+na_n = \rho x_v$.
\item[\textup{(I$4$)}] $na_n = \rho a_m$.
\item[\textup{(I$5$)}] $x_v+(m-1)a_m = \rho a_n$.
\end{itemize}
Using the identities \textup{(I$2$)},\textup{(I$3$)} and \textup{(I$4$)}, we have $x_v = a_q+a_m$. Substituting   $x_v = a_q+a_m$ in   \textup{(I$1$)} and \textup{(I$5$)}, we have  
\begin{itemize}[noitemsep]
\item[\textup{(I$1^*$)}] $qa_q + a_m = \rho a_p$,
\item[\textup{(I$5^*$)}] $a_q + ma_m = \rho a_n$.
\end{itemize}
Without loss of generality if we assume that $a_p = 1$, then 
\begin{itemize}[noitemsep]
\item[\textup{(I$6$)}] $a_q = \dfrac{p}{\rho}$, $a_m = \dfrac{\rho^2 - pq}{\rho}$ and $a_n = \dfrac{\rho^2 - pq}{n}$. 
\end{itemize}
Similarly, if we assume that $a_n = 1$, then
\begin{itemize}[noitemsep]
\item[\textup{(I$7$)}] $a_m = \dfrac{n}{\rho}$, $a_q = \dfrac{\rho^2 - mn}{\rho}$ and $a_p = \dfrac{\rho^2 - mn}{p}$.
\end{itemize}
Moreover, since  the ratio $\dfrac{a_p}{a_n}$ is constant for the Perron vector $X$, so using \textup{(I$6$)} and \textup{(I$7$)}, we have
\begin{itemize}
\item[\textup{(I$8$)}] $pn = (\rho^2 - pq)(\rho^2 - mn)$.
\end{itemize} 

If $m = 1$ and $q > 1$, then by choosing $ a_m = x_v-a_q$, all  the above identities are true. Similarly, for $q = 1$ and $m > 1$, we  choose $ a_q= x_v-a_m$.
\end{obs}

\begin{rem}\label{rem:2.5}
Under the assumption of  Observation~\ref{obs:1}, using  identity \textup{(I8)}   the spectral radius $\rho$ of adjacency matrix $A$ is given by
$$
\rho =  \sqrt{\frac{(pq+mn)+\sqrt{(pq-mn)^2 + 4pn}}{2}},
$$
and hence $\rho > \max \{ \sqrt{ pq} , \sqrt{ mn}\}.$  Observe that, since the spectral radius of  $K_{m,n}$ and  $K_{p,q}$ are  $\sqrt{mn}$ and $\sqrt{pq}$, respectively, so  $\rho > \max \{ \sqrt{ pq} , \sqrt{ mn}\}$ also follows from Lemma~\ref{lem:sr_edge}.
\end{rem}

The next lemma gives us a result on maximal spectral radius among bi-block graphs having two blocks with a fixed number of vertices and independence number.

\begin{lem}\label{lem:two_blocks}
Let $G \in \mathcal{B}(\mathbf{k}, \alpha)$. If $G$  consists of two blocks, then $\rho(G) < \rho(K_{\alpha,\mathbf{k}-\alpha})$. 
\end{lem}

\begin{proof}
Let $G$  be a bi-block graph consists of two blocks $F$ and $H$ connected by  the cut-vertex $v$. Let $F=K(P,Q)$, where $|P|=p$, $|Q|=q$ and $H=K(M,N)$, where $|M|=m$, $|N|=n$ such that 
$Q \cap M = \{v\}$. Then $\mathbf{k}=p+q+m+n-1.$

 If $m = 1$ and $q = 1$, then $\mathbf{k}=p+n+1$  and $G=K_{1, p+n}$ with independence number  $\alpha(G)= p+n$. Thus,   for $\alpha= p+n$ the class $\mathcal{B}(\mathbf{k},\alpha)$   consists of only the star  $G=K_{1, p+n}$ and hence  result is vacuously true. We complete the proof by  considering the following cases.

\underline{\textbf{Case 1:}}  If $p \geq q$ and $n \geq m$, then $\mathcal{I}= P \cup N $ is the $\alpha(G)$-set. We consider the complete bipartite graph $G^*=K(\widetilde{P}, \widetilde{Q})$, where $\widetilde{P}= P \cup N$ and $\widetilde{Q} = Q \cup M$. Thus  $\alpha(G) = \alpha(G^*) = p+n$. Since $G^*$ is obtained from $G$ by adding extra edges, so  by Lemma~\ref{lem:sr_edge}, we have $\rho(G) < \rho(G^*)$.

\underline{\textbf{Case 2:}} If $q>p$ and $m \geq n$, then  $\mathcal{I}= Q \cup M$ is an $\alpha(G)$-set. We consider the complete bipartite graph $G^*=K(\widetilde{P}, \widetilde{Q})$, where $\widetilde{P}= P \cup N$ and $\widetilde{Q} = Q \cup M$. Thus  $\alpha(G) = \alpha(G^*) =  q+m-1$. Since $G^*$ is obtained from $G$ by adding extra edges, so  by Lemma~\ref{lem:sr_edge}, we have $\rho(G) < \rho(G^*)$.

\underline{\textbf{Case 3: }} If $q>p$ and $n>m$, then   $\mathcal{I}=(Q \setminus \{v\}) \cup N$ is an $\alpha(G)$-set and hence $\alpha(G)=q+n-1$. Now we subdivide this case as follows: 

\underline{\textbf{Subcase 3.1:}} Let $p=q-1$. Then $\mathcal{L}=P \cup N$ is an independent set in $G$ and $|\mathcal{L}|=q+n-1$. This implies that $\mathcal{L}$ is also an $\alpha(G)$-set. We consider the complete bipartite graph $G^*=K(\widetilde{P}, \widetilde{Q})$, where $\widetilde{P}= P \cup N$ and $\widetilde{Q} = Q \cup M$. Thus,  $\alpha(G) = \alpha(G^*) =  q+n-1$. Since $G^*$ is obtained from $G$ by adding extra edges, so  by Lemma~\ref{lem:sr_edge}, we have $\rho(G) < \rho(G^*)$.

\underline{\textbf{Subcase 3.2:}} Let $p< q-1$. In view of the  $\alpha(G)$-set $\mathcal{I}=(Q \setminus \{v\}) \cup N$, we 
consider the complete bipartite graph $G^*=K(\widetilde{P}, \widetilde{Q})$, where $\widetilde{P}= P \cup M$ and $\widetilde{Q} = (Q \setminus \{v\}) \cup N$. So $\alpha(G) = \alpha(G^*) = q+n-1$. Observe that, we can obtain the graph $G^*$ from $G$ using the following operations:
\begin{itemize}[noitemsep]
\item[1.] Delete the edges between vertex $v$ and the vertices of $P$.
\item[2.] Add edges between vertices of $M$ and $Q \setminus \{v\}$.
\item[3.] Add edges between vertices of $P$ and $N$.
\end{itemize}
Let $A$ be the adjacency matrix of $G$ and $(\rho,X)$ be the eigen-pair corresponding to the spectral radius of $A$.  Let $A^*$ be the adjacency matrix of $G^*$. Using the notations and identities  in Observation~\ref{obs:1}, we have
\begin{align*}
&\frac{1}{2} X^t(A^*-A)X
=- x_v \sum_{w \in P} x_w + \sum_{\substack{u \sim w \\ u \in M, w \in Q \setminus \{v\}}}x_u x_w + \sum_{\substack{u \sim w \\ u \in P, w \in N }}x_u x_w\\
&= -pa_p(a_q+a_m)+(q-1)a_q(a_q+m a_m)+pna_pa_n & [\mbox{By Eq.}\eqref{eqn:1}] \\
&= -pa_p(a_q+a_m)+(q-1)\rho a_q a_n+pna_pa_n & \text{ [Using (I5*)]}\\
&= -pa_p(a_q+a_m)+(q-1)p a_p a_n+pna_pa_n & \text{[ Using (I2)]}\\
					  &= p\left[(q-1)a_n + \rho a_m - (a_q+a_m)\right] &\text{[Using $a_p = 1$]}\\
					  &= \frac{p}{\rho n} \left[\rho (q-1) (\rho^2 -pq) + \rho n(\rho^2 -pq) - n(p+\rho^2 -pq )\right] & \text{[ Using  (I6)]}\\
					  &= \frac{p}{\rho n} \left[\rho (q-1) (\rho^2 -pq) + \rho n(\rho^2 -pq) - n(\rho^2 -pq) -(\rho^2 - pq)(\rho^2 - mn)\right] & \text{[ Using  (I8)]}\\
					  &= \frac{p(\rho^2 -pq)}{\rho n} \left[\rho (q-1)  + \rho n - n -(\rho^2 - mn)\right]\\
					  &= \frac{p(\rho^2 -pq)}{\rho n} \left[\rho (q+n-1)  -\rho^2 +n(m-1)\right]. 
\end{align*}
By  Lemma~\ref{lem:sr_deg}, we have $\rho \leq \max\{p+n,q\}$. And  using  the assumption $p < q-1$, we always have $q+n-1 > \rho$. Since $X$ is the Perron vector of $G$, so $X^t(A^*-A)X > 0$. Hence by Min-max theorem, we have  $\rho(G) < \rho(G^*)$.

\underline{\textbf{Case  4:}} If $p>q$ and $m>n$, then  $\mathcal{I}= P \cup (M \setminus \{v\})$ is an $\alpha(G)$-set and $\alpha(G)=p+m-1.$ This case is analogous to Case 3 and hence  proceeding similarly, we have $\rho(G) < \rho(G^*)$.
\end{proof}

In the above lemma we have considered  a bi-block graph $G$ with two blocks  and hence the  cut-vertex $v$ have   the block index $bi_{G}(v)=2$. In the  next lemma, we will consider bi-block graphs such that the block index of each of the cut-vertex is exactly $2.$ 
\begin{lem}\label{lem:multi_blocks}
Let $G \in \mathcal{B}(\mathbf{k},\alpha)$. If $bi_{G}(u)=2$ for all cut-vertex $u$ in $G$, then $\rho(G) \leq \rho(K_{\alpha, \mathbf{k} - \alpha})$ and equality holds if and only if $G = K_{\alpha, \mathbf{k}- \alpha}$. 
\end{lem}

\begin{proof}
We will use induction on the number of blocks to prove the lemma.  Let $G \in \mathcal{B}(\mathbf{k},\alpha)$ be  a bi-block graph that consists of $b$  blocks and $bi_{G}(c)=2$ for every cut-vertex $c$ in $G$. By Lemma~\ref{lem:two_blocks}, the result is true for $b=2.$ We assume that the result is true for all bi-block graphs in $\mathcal{B}(\mathbf{k},\alpha)$ consisting of $b-1$  blocks.  Let $H=K(M,N)$ with $|M|=m$ and $|N|=n$ be a leaf block connected to the graph $G$ at a cut-vertex $v$. Since  $bi_{G}(v)=2$, so there exists a unique block $F=K(P,Q)$ with $|P|=p$ and $|Q|=q$  which is a neighbour of $H$ connected via the cut-vertex $v$. Without loss of generality, we assume that $M \cap Q = \{v\}$. Let $\mathcal{I}$ be an $\alpha(G)$-set of $G$, {\it{i.e.}}, $|\mathcal{I}| = \alpha$.

\underline{\textbf{Case 1:}} $\mathcal{I} \cap P =  \emptyset$  and $\mathcal{I} \cap Q =  \emptyset$. In this case, either   $M \setminus\{v\} \subset \mathcal{I}$  or $N \subset \mathcal{I}$.   We consider the complete bipartite graph $K(\widetilde{P}, \widetilde{Q})$, where $\widetilde{P} = P \cup N$ and $\widetilde{Q} = Q \cup M$. Let $G^*$ be the graph obtained from $G$ by replacing the induced subgraph $F\circledcirc H$  with  $K(\widetilde{P}, \widetilde{Q})$. Then, the resulting graph $G^*$ consists of $b-1$ blocks and $\mathcal{I}$  is an $\alpha(G^*)$-set, {\it i.e.,} $G^* \in \mathcal{B}(\mathbf{k},\alpha)$. Since $G^*$ is obtained from $G$ by adding additional edges, so by Lemma~\ref{lem:sr_edge}, we have $\rho(G) < \rho(G^*)$. Thus, the induction hypothesis yields the result.

\underline{\textbf{Case 2:}} $\mathcal{I} \cap P =  \emptyset$  and $\mathcal{I} \cap Q \neq \emptyset$. For   $m \geq n$,  we can assume  $ M \subset\mathcal{I}$.  We consider graph  $G^*$ which is obtained from $G$ by replacing the induced subgraph $F\circledcirc H$  with  $K(\widetilde{P}, \widetilde{Q})$, where $\widetilde{P} = P \cup N$ and $\widetilde{Q} = Q \cup M$,  which implies that  $\mathcal{I}$  is an $\alpha(G^*)$-set. Thus arguing similar to the Case 1  yields the result.

\underline{\textbf{Case 3:}}  $\mathcal{I} \cap P =  \emptyset$  and $\mathcal{I} \cap Q \neq \emptyset$. For $n > m$, if  $v \in \mathcal{I}$, then $\mathcal{L}=(\restr{\mathcal{I}}{G-H}\setminus \{v\}) \cup N$ is an independent set of $G$ and $|\mathcal{L}|> |\mathcal{I}|$, which  leads to a contradiction. Thus $v \notin \mathcal{I}$ and we have the following:
\begin{equation}\label{eqn:2}
\begin{cases}
 v \notin \mathcal{I} \mbox { and } \mathcal{I}= \restr{\mathcal{I}}{G-H} \cup N, \\
 \alpha(G)= \left|\restr{\mathcal{I}}{G-H}\right| +n.
 \end{cases}
\end{equation}
Next, we subdivide the case  $\mathcal{I} \cap P =  \emptyset$ and $\mathcal{I} \cap Q \neq \emptyset$ with $n>m$, into the following subcases.

\underline{\textbf{Subcase 1:}}  Suppose that  all the vertices of $Q$ are cut-vertices. Let $u \in Q \setminus \{v\}$ be a cut-vertex and $u \in \mathcal{I}$.  Since $bi_{G}(u)=2$, so  let $B=K(R,S)$ be the   neighbour of the block $F$ via the cut-vertex $u$, where $ R\cap Q= \{u\}$. Thus,  $u \in \mathcal{I}$  and $u\in R$  implies  that $\mathcal{I} \cap S = \emptyset$.  Consider the bi-block graph $G^*$ obtained from $G$ by replacing the induced subgraph $F\circledcirc B$  with the complete bipartite graph $K(\widetilde{P}, \widetilde{Q})$, where $\widetilde{P} = P \cup S$ and $\widetilde{Q} = Q \cup R$. It is easy to see that  $\mathcal{I}$  is an $\alpha(G^*)$-set and  $G^*\in \mathcal{B}(\mathbf{k}, \alpha)$  consists of $b-1$ blocks. Hence the result follows from the Lemma~\ref{lem:sr_edge} and the induction hypothesis.

\underline{\textbf{Subcase 2:}}   Let $c\in Q$   and $c$ is not a cut-vertex. Since $\mathcal{I} \cap P =\emptyset$, so  $c \in \mathcal{I}$. Let $A$ be the adjacency matrix of $G$ and $(\rho,X)$ be the eigen-pair corresponding to the spectral radius of $A$.  Let $x_u$ denote the entry  of $X$ corresponding to  the vertex $u\in V$. Using $AX=\rho X$ and arguing  similar to the Observation~\ref{obs:1},  we  find a few identities as follows. For $m\geq 2$,  let us denote
\begin{equation}\label{eqn:3}
x_u=
\begin{cases}
b_n & \mbox{ if } u\in N,\\
b_m  & \mbox{ if } u\in M, u\neq v.
\end{cases}
\end{equation}

Using $c\in Q$, $c$ is not a cut-vertex  and $AX=\rho X$, we have the following identities:
\begin{itemize}[noitemsep]
	\item[\textup{(J$1$)}] $\rho x_c = \sum_{w \in P} x_w$.
	\item[\textup{(J$2$)}] $\rho x_v = \sum_{w \in P} x_w + n b_n$.
	\item[\textup{(J$3$)}] $\rho b_n = (m-1)b_m + x_v$.
	\item[\textup{(J$4$)}] $\rho b_m =  n b_n$.
\end{itemize}
Using indentities \textup{(J$1$)}, \textup{(J$2$)} and \textup{(J$4$)}, we have $x_v = x_c + b_m$. Thus the identity \textup{(J$3$)}  reduces to:
	\begin{itemize}[noitemsep]
	 	\item[\textup{(J$3^*$)}] $\rho b_n = mb_m + x_c$.
	\end{itemize}
 Next, if $m=1$, then by choosing $b_m = x_v - x_c$, all the above identities  are true. Now  we further subdivide the Subcase 2 as follows: 
 
\underline{\textbf{Subcase 2.1:}} Whenever $b_m \geq  b_n$.\\ 
Let  $G^*$ be a bi-block graph obtained from $G$ by replacing  the induced subgraph $F\circledcirc H$  with the complete bipartite graph $K(\widetilde{P}, \widetilde{Q})$, where $\widetilde{P} =  P \cup M$ and $\widetilde{Q} = (Q \setminus \{v\}) \cup N$. Thus, $\mathcal{I}$ is an $\alpha(G^*)$-set and  $G^*\in \mathcal{B}(\mathbf{k}, \alpha)$  consists of $b-1$ blocks. Note that,  we can obtain the graph $G^*$ from $G$ using the following operations:
\begin{itemize}[noitemsep]
\item[1.] Delete the edges between vertex $v$ and the vertices of $P$.
\item[2.] Add edges between vertices of $M$ and $Q \setminus \{v\}$.
\item[3.] Add edges between vertices of $P$ and $N$.
\end{itemize}
Let $A^*$ be the adjacency matrix of $G^*$. Using  the above identities,  we have
\begin{align*}
&\frac{1}{2} X^t(A^*-A)X 
= -x_v  \sum_{w \in P} x_w + \sum_{\substack{u \sim w \\ u \in M, w \in Q \setminus \{v\}}}x_u x_w + \sum_{\substack{u \sim w \\ u \in P, w \in N }}x_u x_w\\
&= -(x_c + b_m)  \sum_{w \in P} x_w + (mb_m + x_c) \sum_{w \in Q \setminus \{v\}} x_w +  n b_n \sum_{w \in P} x_w & [\mbox{By Eq.}\eqref{eqn:3}]\\
&= -(x_c + b_m)  \rho x_c + (mb_m + x_c) \sum_{w \in Q \setminus \{v\}} x_w +  n b_n  \rho  x_c & \text{[Using \textup{(J$1$)}]}\\
 &= -(x_c + b_m)  \rho x_c + (mb_m + x_c) \sum_{w \in Q \setminus \{v\}} x_w +  \rho^2 b_m  x_c & \text{[Using \textup{(J$4$)}]}\\
					  &\geq -(x_c + mb_m)  \rho x_c + (mb_m + x_c) \sum_{w \in Q \setminus \{v\}} x_w +  \rho^2 b_m x_c\\
					  &= -\rho^2 b_nx_c + (mb_m + x_c) \sum_{w \in Q \setminus \{v\}} x_w +  \rho^2 b_m  x_c & \text{[Using \textup{(J$3^*$)}]}\\
					   &= \rho^2(b_m -  b_n)x_c + (mb_m +  x_c) \sum_{w \in Q \setminus \{v\}} x_w.
\end{align*}
 Since $b_m \geq  b_n$,  and $X$ is the Perron vector of $G$, so $X^t(A^*-A)X \geq 0$. Thus, by Min-max theorem, we have  $\rho(G) \leq \rho(G^*)$ and hence the induction hypothesis yields the result.
	
\underline{\textbf{Subcase 2.2:}} Whenever $b_m <  b_n$.\\
For  this case we partition the set $N \subset \mathcal{I}$ as $N=N_1 \cup N_2$ and $N_1 \cap N_2=\emptyset$ such that $|N_1| = m$ and $|N_2| = n-m$. We consider the complete bipartite graph $K(\widetilde{P}, \widetilde{Q})$, where $\widetilde{P} = P \cup N_1$ and $\widetilde{Q} = Q \cup M \cup N_2$. Let  $G^*$ be a bi-block graph obtained from $G$ by replacing  the induced subgraph $F\circledcirc H$  with  $K(\widetilde{P}, \widetilde{Q})$. Thus, by Eq.~\eqref{eqn:2}, we obtain that $\mathcal{I}^*=\restr{\mathcal{I}}{G-H}\cup  M \cup N_2 $  is an $\alpha(G^*)$-set and   $\alpha(G^*)= \alpha(G)=\left|\restr{\mathcal{I}}{G-H}\right| +n,$ which implies that  $G^*\in \mathcal{B}(\mathbf{k}, \alpha)$  consists of $b-1$ blocks. Note that, we can obtain the graph $G^*$ from $G$ using the following operations:
\begin{itemize}[noitemsep]
\item[1.] Delete the edges between vertices of  $M$ and  $N_2$.
\item[2.] Add edges between vertices of $N_1$ and $Q \setminus \{v\}$.
\item[3.] Add edges between vertices of $P$ and $N_2$.
\item[4.] Add edges between vertices of $N_1$ and $N_2$.
\item[5.] Add edges between vertices of $M\setminus \{v\}$ and  $P$.
\end{itemize}
Let $A^*$ be the adjacency matrix of $G^*$. Then,
\begin{align*}
& \frac{1}{2} X^t(A^*-A)X = -\sum_{\substack{u \sim w \\ u \in M, w \in N_2}}x_u x_w + \sum_{\substack{u \sim w \\ u \in N_1, w \in Q \setminus \{v\}}}x_u x_w +  \sum_{\substack{u \sim w \\ u \in N_2, w \in P }}x_u x_w \\ &\hspace*{7cm} +  \sum_{\substack{u \sim w \\ u \in N_1, w \in N_2 }}x_u x_w + \sum_{\substack{u \sim w \\ u \in M\setminus \{v\}, w \in P }}x_u x_w\\
&= -(n-m)(mb_m + x_c) b_n + mb_n \sum_{w \in Q \setminus \{v\}} x_w + (n-m) b_n \sum_{w \in P} x_w \\
&\hspace*{7cm} + (n-m) m b_n^2  + b_m(m-1) \sum_{w \in P}  x_w  & [\mbox{By Eq.}\eqref{eqn:3}] \\
					  &= -(n-m)mb_mb_n - (n-m)x_c b_n + mb_n \sum_{w \in Q \setminus \{v\}} x_w + \rho (n-m) b_n x_c \\
					  &\hspace*{7.6cm} + (n-m) m b_n^2 +\rho (m-1)b_m x_c &\text{[Using \textup{(J$1$)}]}\\
					  &= (n-m)\left[ mb_n(b_n-b_m) + (\rho - 1)x_c b_n\right] + mb_n \sum_{w \in Q \setminus \{v\}} x_w +\rho (m-1)b_m x_c .
\end{align*}
Since $b_m <  b_n$ and $\rho  \geq 1$ (by Lemma~\ref{lem:sr_deg}), so using the fact that  $X$ is the Perron  vector of $G$, we have   $X^t(A^*-A)X \geq  0$. Thus, by Min-max theorem, we have  $\rho(G) \leq \rho(G^*)$ and hence the induction hypothesis yields the result.

\underline{\textbf{Case 4:}}  $\mathcal{I} \cap P \neq \emptyset$ and   $\mathcal{I} \cap Q = \emptyset$. For $n \geq m$ or $m=n+1$ , we have  $N \subset \mathcal{I}$. We consider graph  $G^*$ obtained from $G$ by replacing  the induced subgraph $F\circledcirc H$  with  $K(\widetilde{P}, \widetilde{Q})$, where $\widetilde{P} = P \cup N$ and $\widetilde{Q} = Q \cup M$,  which implies that  $\mathcal{I}$  is an $\alpha(G^*)$-set. Thus, arguments similar to the Case $1$  yields the result.

\underline{\textbf{Case 5:}}  $\mathcal{I} \cap P \neq \emptyset$ and   $\mathcal{I} \cap Q = \emptyset$. For   $m >n+1$, we have   $(M\setminus \{v\}) \subset \mathcal{I}$.  We consider all neighbouring blocks of $F=K(P,Q)$, say $B_i=K(R_i,S_i)$ for $1\leq i\leq j$,  connected via cut-vertices to the vertex partition  $P$. Without loss of generality, we assume  $S_i \cap P \neq \emptyset$. For any one of the such   neighbour, if  $\mathcal{I} \cap R_i = \emptyset$, then we consider the graph  $G^*$ which is  obtained from $G$ by replacing  the induced subgraph $F\circledcirc B_i$  with  $K(\widetilde{P}, \widetilde{Q})$, where $\widetilde{P} = P \cup S_i$ and $\widetilde{Q} = Q \cup R_i$. Since $\mathcal{I} \cap P\neq \emptyset$, so  $\mathcal{I}$  is an $\alpha(G^*)$-set and argument similar to the Case 1  leads to the desired  result. If no such neighbours exists, then proceeding inductively we need to look for $B_i$'s neighbours   with similar properties. Since $G$ is a finite graph, either we will reach a neighbour with suitable properties or reach a leaf block does not satisfies requisite properties. For the later case, we find a finite chain of blocks $C_i=K(M_i,N_i)$ for $1\leq i\leq t$ satisfying the following:
\begin{itemize}[noitemsep]
\item[1.] $C_1=H$ and $C_t$ are leaf blocks.
\item[2.] For $i = 1,2,\cdots,t-1$ , the blocks $C_i$ and $C_{i+1}$ are neighbours 
          such that $M_i \cap N_{i+1} \neq \emptyset$.
\item[3.] $\mathcal{I} \cap N_i = \emptyset$ for all $i=1,2,\ldots,t$.
\end{itemize}
Since $C_t$ is a leaf block and is connected to $C_{t-1}$ via a cut-vertex $u(say)$ with $bi_G(u)=2$, so  it can be seen  $\mathcal{I} \cap N_{t-1} = \emptyset$ and $\mathcal{I} \cap N_t = \emptyset$ implies that  $|M_t|> |N_t|$. Now, if  we begin with the leaf block $C_t$, then  this case is analogous to the Case 3. Hence the desired result follows.

Moreover, by  Lemma~\ref{lem:two_blocks} and combining all the above cases, the maximum spectral radius $\rho(G)$ among  all graphs $G$ in $\mathcal{B}(\mathbf{k}, \alpha)$  with $bi_{G}(u)=2$ for all cut-vertex $u$ in $G$ is uniquely attained for the complete bipartite graph $K_{\alpha, \mathbf{k}-\alpha}$.
\end{proof}

\begin{lem}\label{lem:star}
If $G \in \mathcal{B}(\mathbf{k},\alpha)$, then there exists a bi-block graph  $G^* \in \mathcal{B}(\mathbf{k},\alpha)$ with $bi_{G^*}(u)=2$ for all cut-vertex $u$ in $G^*$ such that $\rho(G) \leq \rho(G^*)$.
\end{lem}

\begin{proof}
Let $v$ be a cut-vertex of $G$ with $bi_G(v) = t$, where $t \geq 3$. Let $B_i = K(M_i,N_i)$;   $i=1,2,3$ be any three neighbours connected via the cut-vertex $v$ such that $v \in N_1 \cap N_2 \cap N_3$. Let $\mathcal{I}$ be an  $\alpha(G)$-set.  If $V(B_i)\cap \mathcal{I}\neq \emptyset$ for all  $i=1,2,3$, then either $M_i\cap  \mathcal{I}\neq \emptyset $ or $N_i\cap  \mathcal{I}\neq \emptyset $. Thus by pigeonhole principle,  there exist $ i,j \in \{1,2,3\}$ such that either $\mathcal{I} \cap N_i = \emptyset$ and $\mathcal{I} \cap N_j = \emptyset$ or  $\mathcal{I} \cap M_i = \emptyset$ and $\mathcal{I} \cap M_j = \emptyset$. Let us consider  a bi-block  graph $G^*$   obtained from $G$ by replacing the induced subgraph $B_i \circledcirc B_j$ with  $K( \widetilde{M}, \widetilde{N})$, where $\widetilde{M} = M_i \cup M_j$ and $\widetilde{N} = N_i \cup N_j$. It is easy to see that, $\mathcal{I}$ is an $\alpha(G^*)$-set and $bi_{G^*} (v) = t-1$. By Lemma~\ref{lem:sr_edge}, we have $\rho(G) \leq \rho(G^*)$. Hence proceeding inductively the  result follows.
If  $V(B_{i_0})\cap \mathcal{I} = \emptyset$ ({\it i.e.} $M_{i_0}\cap  \mathcal{I} = \emptyset $ and $N_{i_0}\cap  \mathcal{I} = \emptyset $) for some  $i_0 \in \{1,2,3\}$, then  for  $j\neq {i_0}$ and choosing $K( \widetilde{M}, \widetilde{N})$, where $\widetilde{M} = M_{i_0} \cup M_j$ and $\widetilde{N} = N_{i_0} \cup N_j$, similar argument yields the desired result.
\end{proof}

Next we state the main result of the article (without proof) which maximizes the spectral radius  for the class   $ \mathcal{B}(\mathbf{k},\alpha)$ and the proof follows from Lemmas~\ref{lem:multi_blocks} and~\ref{lem:star}.

\begin{theorem}
If $G \in \mathcal{B}(\mathbf{k},\alpha)$, then $\rho(G) \leq \rho(K_{\alpha, \mathbf{k} - \alpha})$ and equality holds if and only if $G = K_{\alpha, \mathbf{k} - \alpha}$.
\end{theorem}

\noindent{ \textbf{\Large Acknowledgements}}: We sincerely thank two anonymous referees for their critical
 reading of the manuscript and for several helpful suggestions which have immensely helped us in getting the article to its present form.  Sumit Mohanty would like to thank the Department of Science and Technology, India, for financial support through the projects MATRICS (MTR/2017/000458).

\small{

}


\begin{thebibliography}{20}


\bibitem{Bapat}   Bapat R B,  Graphs and matrices. Second Edition, Hindustan Book Agency, 
New Delhi, (2014).

\bibitem{Cri}  Conde C M,  Dratman E and  Grippo  L N, On the spectral radius of block graphs with prescribed independence number $\alpha$. Linear Algebra Appl., (2020).

\bibitem{Feng} Feng L and  Song J,
Spectral radius of unicyclic graphs with given independence number,
Util. Math., {\bf 84} (2011),  33-43.

\bibitem{Guo} Guo J M and Shao J Y,
On the spectral radius of trees with fixed diameter,
Linear Algebra Appl., {\bf 413} (2006),  131-147.

\bibitem{Chu}  Ji C and  Lu M,
On the spectral radius of trees with given independence number,
Linear Algebra Appl., {\bf 488} (2016),  102-108.

\bibitem{Li} Li, Q; Feng, K Q, On the largest eigenvalue of a graph. (Chinese) Acta Math. Appl. Sinica 2, {\bf 2}  (1979), 167-175.

\bibitem{Lui}  Liu H,   Lu M and  Tian F,
On the spectral radius of graphs with cut edges,
Linear Algebra Appl.,{\bf 389} (2004),  139-145.

\bibitem{Lu}  Lu H and  Lin Y,
Maximum spectral radius of graphs with given connectivity, minimum degree and independence number,
J. Discret. Algorithms, {\bf 31} (2015),  113-119.



\bibitem{Xiao}  Wu B,  Xiao E and  Hong Y,
The spectral radius of trees on k pendant vertices,
Linear Algebra Appl., {\bf 395} (2005),  343-349.



\bibitem{Xu} Xu M, Hong  Y,  Shu J and Zhai M,
The minimum spectral radius of graphs with a given independence number,
Linear Algebra Appl., {\bf 431(5)} (2009), pp. 937-945.





\end{thebibliography}
\end{document}